\documentclass[11pt,a4paper,reqno]{amsart}

\numberwithin{equation}{section}

\usepackage[top=3cm, bottom=3cm, left=3cm, right=3cm]{geometry}
\usepackage[latin1]{inputenc}
\usepackage{csquotes}
\usepackage{amsmath}
\usepackage{amsthm}
\usepackage{amsfonts}
\usepackage{amssymb}
\usepackage{graphics}
\usepackage{float}
\usepackage[hang,flushmargin]{footmisc} 

\usepackage{amsmath,amsthm,amssymb,graphicx, multicol, array}
\usepackage{enumerate}
\usepackage{enumitem}

\usepackage{xcolor}

\usepackage[pagebackref,bookmarks,colorlinks,breaklinks]{hyperref}

\hypersetup{linkcolor=blue,citecolor=red,filecolor=blue,urlcolor=blue} 

\usepackage{tabto}

\numberwithin{equation}{section}

\usepackage{tikz}
\usepackage{pgfplots}

\usetikzlibrary{matrix}

\usepackage{mathrsfs}

\DeclareMathOperator{\ord}{ord}

\newtheorem{thm}{Theorem}[section]
\newtheorem{lem}{Lemma}[section]
\newtheorem{conj}{Conjecture}[section]

\newtheorem{dfn}{Definition}[section]

\newcommand{\N}{\mathbb{N}}
\newcommand{\Z}{\mathbb{Z}}

\usepackage{fancyhdr}
\begin{document}

\title{A Result for Germain Primes}


\author{N. A. Carella}
\address{}
\curraddr{}
\email{}
\thanks{}




\date{}

\dedicatory{}

\begin{abstract}This article determines a lower bound for the number Germain prime pairs $p$ and $2p+1$ up to a large number $x$.
\end{abstract}

\maketitle
\tableofcontents
\pagenumbering{Page gobble}
\pagenumbering{arabic}
\section{\textbf{Introduction and the Main Result}}\label{S3030}
The distribution of the sequence of Germain primes $(2,5), (3,7), (5,11)\ldots, (p,2p+1), \ldots, $ and the distributions of other sequences of linearly dependent pairs $p$ and $p<q$ such that $ap+bq=c$, and prime $k$-tuples are long standing topics of research in number theory. Discussions on the prime pairs problems appear in \cite{RP1996}, \cite{NW2000}, and many other references in the vast literature on this subject.\\ 

The predicted asymptotic form for the Germain primes problem is given below.
\begin{conj} \label{conj3030.900}{\normalfont (\cite[Conjecture D, p. 45]{HL1923})} There are infinitely many Germain prime pairs. If $\pi_{\mathcal{G}}(x)$ is the number of pairs less than $x$, then 
	$$\pi_{\mathcal{G}}(x)=2C_2\int_2^x\frac{1}{\log (t)\log(2t+1)}dt+O\left( \frac{x}{(\log x)^3}\right), $$ 
	where $C_2>0$ is a constant defined by
\begin{equation}\label{eq3030.900}
C_2=\prod_{p\geq3}\left(1-\frac{1}{(p-1)^2} \right)=0.6601618605898407646766938915352060
	\ldots. 
\end{equation}
\end{conj} 

More generally, Conjecture D handles linearly dependent prime pairs $ap+bq=c$, where $a,b,c\in \Z^{\times}$ are constants. Some numerical data to back the Germain primes conjecture are available in the literature, see \cite[Section 3.5]{CC2021}, and \cite[p.\ 4]{WS2021}.\\ 

Let $\Lambda(n)$ denotes the weighted prime power indicator function, (von Mangoldt function),  
\begin{equation}\label{eq3030.050}
	\Lambda(n)=
	\begin{cases}
		\log n &\text{ if } n=p^k,\\
		0&\text{ if } n\ne p^k.\\
	\end{cases}
\end{equation}
The conjecture is equivalent to the weighted sum \begin{equation}\label{eq3030.060}
	\sum_{ n \leq x} \Lambda(n)\Lambda(2n+1)=2C_2x+o(x).
\end{equation} This note proposes a weaker asymptotic formula.
\begin{thm}\label{thm3030.900}
	If $x\geq 1$ is a large real number, then
	\begin{equation}
		\sum_{ n\leq x}\Lambda(n)\Lambda(2n+1) \gg \frac{x\log \log x}{(\log x)(\log \log \log x)}\nonumber.
	\end{equation}
\end{thm}
This is not to far from the conjecture asymptotic in \eqref{eq3030.060}, and the unconditional upper bound
\begin{equation}\label{eq3030.070}
	\sum_{ n\leq x}\Lambda(n)\Lambda(2n+1) \leq cx ,
\end{equation}
where $c>0$ is an explicit constant, see \cite[Corollary 1.2]{BS2022}.\\

A short outline of this article is provided here. Theorem \ref{thm3030.900} is a simple corollary of Theorem \ref{thm2020.700} in Section \ref{S2020}. The basic materials required to prove the fundamental result in Theorem \ref{thm2020.700} are developed and proved in Section \ref{S4040} to Section \ref{S2059}.  Section \ref{S4040} deals with several forms of the finite sum $\sum_{m,n\leq x}f(m,n)/[m,n]$, which are of independent interest in number theory. An asymptotic formula for the uniform distribution of integers in arithmetic progressions is proved in Section \ref{A2002}. The proof Theorem \ref{thm3030.900} of appears in Section \ref{S3535}.

\section{\textbf{Foundational Results}}\label{S4040}
The expressions $(m,n)=\text{gcd}(m,n)$ and $[m,n]=\text{lcm}(m,n)$ denote the greatest common divisor and the lowest common multiple respectively. The totient function is defined by 
\begin{equation}\label{eq4040.020}
	\varphi(n)=n\prod_{p\mid n}\left( 1-\frac{1}{p}\right) ,
\end{equation} and the Mobius function is defined by
\begin{equation}\label{eq4040.030}
	\mu(n)=
	\begin{cases}
		(-1)^w &\text{ if } n=p_1p_2\cdots p_w,\\
		0&\text{ if } n\ne p_1p_2\cdots p_w.
	\end{cases}
\end{equation}

The nonnegativity of the finite sum
\begin{equation}\label{eq4040.040}
	\sum_{m,n \leq x} \frac{\mu(m)\mu(n)}{[m,n]}	>0
\end{equation}
and the convergence of the associated series as $x\to\infty$ is the subject a study in \cite{DI1983}, and in sieve theory. Similar techniques are used here to derive several estimates and verify the nonnegativity of some related finite sums. These finite sums arise in the analysis of the main term and error term of Theorem \ref{thm2020.700}.

\subsection{Elementary Identities}

\begin{lem} \label{lem4040.600E} If $m,n  \geq 1$ are a pair of integers, then,
	\begin{equation} \label{eq4040.600E}
		\gcd(m,n)=\sum_{d\mid (m,n)}\varphi(d).\nonumber
	\end{equation}
\end{lem}

\begin{proof}[\textbf{Proof}] The claim follows from the additive to multiplicative relation
	\begin{eqnarray} \label{eq4040.610E}
		\sum_{d\mid (m,n)}\varphi(d)&=&\prod_{p^v\mid \mid (m,n)}\left(1+ \varphi(p)+\varphi(p^2)+\cdots +\varphi(p^v)\right) \\[.2cm]
		&=&\prod_{p^v\mid \mid (m,n)}p^v\nonumber \\[.2cm] &=&\gcd(m,n)\nonumber,
	\end{eqnarray}	
	where $p^v\mid \mid (m,n)$ is the maximal prime power divisor.
\end{proof}

\begin{lem} \label{lem4040.625E} If $m,n  \geq 1$ are a pair of integers, then,
	\begin{equation} \label{eq4040.625E}
		\frac{1}{[m,n]}=\frac{1}{mn}\sum_{d\mid (m,n)}\varphi(d).\nonumber
	\end{equation}
\end{lem}
\begin{proof}[\textbf{Proof}] Use Lemma \ref{lem4040.600E}, to transform the denominator as follows.
	\begin{equation}\label{eq4040.630E}
		\frac{1}{[m,n]}
		=\frac{\gcd(m,n)}{mn} =\frac{1}{mn}\sum_{d\mid (m,n)}\varphi(d).
	\end{equation}
\end{proof}
\begin{lem} \label{lem4040.650E} If $m,n  \geq 1$ are a pair of integers, then,
	\begin{equation} \label{eq4040.650E}
		\frac{1}{\varphi([m,n])}=\frac{1}{\varphi(mn)}\sum_{d\mid (m,n)}\varphi(d).\nonumber
	\end{equation}
\end{lem}

\begin{proof}[\textbf{Proof}] Substitute the identity \eqref{eq4040.020} to transform the denominator as	 follows.
	\begin{eqnarray}\label{eq4040.270M}
		\frac{1}{\varphi([m,n])}&=&\frac{1}{[m,n]}\prod_{p\mid [m,n]}\left( 1-\frac{1}{p}\right)^{-1}\\[.2cm]
		&=&\frac{\gcd(m,n)}{mn}\prod_{p\mid mn}\left( 1-\frac{1}{p}\right)^{-1} \nonumber\\[.2cm]
		&=&\frac{1}{\varphi(mn)} \sum_{d\mid (m,n)}\varphi(d)\nonumber .
	\end{eqnarray}
	The reverse the identity \eqref{eq4040.020} is used on the penultimate line of equation \eqref{eq4040.270M}, and Lemma \ref{lem4040.600E} is used to obtain the last line.
\end{proof}

\subsection{Elementary Estimates}

\begin{lem} \label{lem4040.800L} If $x \geq 1$ is a large number, then,
	\begin{equation} \label{eq4040.800L}
		\sum_{m,\;n \leq x} \frac{\log m\log n}{[m,n]}\ll(\log x)^5\nonumber
	\end{equation}
	as $x\to \infty$.
\end{lem}
\begin{proof}[\textbf{Proof}] Use Lemma \ref{lem4040.625E} and switch the order of summation to obtain 
	\begin{eqnarray}\label{eq4040.820L}
		\sum_{m,\;n \leq x} \frac{\log m\log n}{[m,n]}&=&	\sum_{m,n \leq x} \frac{\log m\log n}{mn}\sum_{d\mid (m,n)}\varphi(d) \\[.2cm]
		&=&	\sum_{d\leq x}\varphi(d)\sum_{\substack{m,n \leq x\\d\mid (m,n)}} \frac{\log m\log n}{mn}\nonumber .
	\end{eqnarray}
	Let $m=dr$ and $n=ds$, where $\gcd(r,s)=1$.	Substituting this change of variables, and simplifying yield the upper bound
	\begin{eqnarray}\label{eq4040.830L}
		\sum_{d\leq x}\varphi(d)\sum_{\substack{m,n \leq x\\d\mid (m,n)}} \frac{\log m\log n}{mn}
		&=&	\sum_{d\leq x}\frac{\varphi(d)}{d^2}\sum_{r,s \leq x/d} \frac{\log r\log s}{rs}\\
		&\ll&(\log x)^4	\sum_{d\leq x}\frac{\varphi(d)}{d^2} \nonumber\\ 
		&\ll&(\log x)^5\nonumber,
	\end{eqnarray}
	where $\sum_{d\leq x}\varphi(d)d^{-2}\ll \log x$.
\end{proof}

\begin{lem} \label{lem4040.250P} If $x \geq 1$ is a large number, then,
	\begin{equation} \label{eq4040.250P}
		\sum_{m,\;n \leq x} \frac{\mu(m)\mu(n)\log m\log n}{\varphi([m,n])}\gg\frac{\log x}{\log \log x}\nonumber
	\end{equation}
	is an increasing nonnegative function as $x\to \infty$, and bounded by $\log x$.
\end{lem}
\begin{proof}[\textbf{Proof}] By Lemma \ref{lem4040.650E}, the finite sum transforms as	
	\begin{eqnarray}\label{eq4040.260P}
		B(x)&=&		\sum_{m,\;n \leq x} \frac{\mu(m)\mu(n)\log m\log n}{\varphi([m,n])}\\&=&	\sum_{m,n \leq x} \frac{\mu(m)\mu(n)\log m\log n}{\varphi(mn)}\sum_{d\mid (m,n)}\varphi(d)\nonumber \\[.2cm]
		&=&	\sum_{d\leq x}\varphi(d)\sum_{\substack{m,n \leq x\\d\mid (m,n)}} \frac{\mu(m)\mu(n)\log m\log n}{\varphi(mn)}\nonumber.
	\end{eqnarray}
	Replace the change of variables $m=dr$ and $n=ds$, where $r,s\geq 1$ are squarefree integers such that $\gcd(r,s)=1$ to obtain the expression
	\begin{eqnarray}\label{eq4040.265P}
		B(x)&=&	\sum_{d\leq x}\frac{\varphi(d)\mu^2(d)}{\varphi(d^2)}\sum_{\substack{r,s \leq x/d\\(d,rs)=1\\\gcd(r,s)=1}} \frac{\mu(r)\mu(s)\log dr\log ds}{\varphi(rs)} \\
		&=&	\sum_{d\leq x}\frac{\mu^2(d)}{d}\left(\sum_{\substack{r \leq x/d\\(d,r)=1}} \frac{\mu(r)\log dr}{\varphi(r)}\right)^2  \nonumber \\
		&>&0\nonumber.
	\end{eqnarray}
Next, the asymptotic formula for the inner sum given in Lemma \ref{lem4040.800P} leads to 
\begin{eqnarray}\label{eq4040.280P}
B(x)	&=&	\sum_{d\leq x}\frac{\mu^2(d)}{d}\left(\sum_{\substack{r \leq x/d\\(d,r)=1}} \frac{\mu(r)\log dr}{\varphi(r)}\right)^2 \\
&=&	\sum_{d\leq x}\frac{\mu^2(d)}{d}\left(\mathfrak{G}(d)+O\left(e^{-b\sqrt{\log x} }\right)\right)^2 \nonumber\\
&=&	\sum_{d\leq x}\frac{\mu^2(d)\mathfrak{G}^2(d)}{d}+O\left((
\log x)e^{-b_1\sqrt{\log x} }\right) \nonumber.
\end{eqnarray}
The singular series $\mathfrak{G}(d)>1$ is an absolutely convergent series, see \eqref{eq4040.290P} below, and has the asymptotic form specified in \eqref{eq4040.295P}. Together, these estimates yield the lower bound
\begin{eqnarray}\label{eq4040.285P}
	B(x)	
	&\gg&	\sum_{d\leq x}\frac{\mu^2(d)}{d}+O\left(e^{-b_2\sqrt{\log x} }\right) \\
	&\gg&\frac{\log x}{\log \log x}
	+O\left(e^{-b_2\sqrt{\log x} }\right)\nonumber \\	&\gg&\frac{\log x}{\log \log x}\nonumber
\end{eqnarray}
since  
\begin{equation}\label{eq4040.270P}
	\frac{\log x}{\log \log x}\ll\sum_{d\leq x}\frac{\mu^2(d)}{d}=\frac{6}{\pi^2}\log x+O\left( \frac{1}{\log x}\right).
\end{equation} 
Hence, the product $B(x)$ satisfies the inequality
	\begin{equation}\label{eq4040.287P}
		\frac{\log x}{\log \log x}\ll		\sum_{m,\;n \leq x} \frac{\mu(m)\mu(n)\log m\log n}{\varphi([m,n])}\ll\log  x,
	\end{equation} 
	which is an unbounded function as $x\to \infty$.
\end{proof}
\subsection{The Singular Series for Prime Pairs}
For each integer $d\geq1$ the singular series
\begin{equation}\label{eq4040.290P}
\mathfrak{G}(d)\begin{cases}
=0&\text{ if }d=2m+1,\\
> 1&\text{ if }d=2m, 
\end{cases}	
\end{equation}
is a small constant $>1$. More precisely, for $m\geq2$, the singular series is given by the infinite product 
\begin{equation}\label{eq4040.295P}
	\mathfrak{G}(2m)=2C_2\prod_{2<p\mid m}\left( \frac{p-1}{p-2}\right) >1.
\end{equation}
The first case $\mathfrak{G}(2)=2C_2>1$ is called the twin prime constant, see Conjecture \ref{conj3030.900}. \\

\begin{lem} \label{lem4040.850P} {\normalfont(\cite[Lemma 2.1]{GC2001}) }Let $\geq2$ be a fixed integer, and let
	$x \geq 1$ be a large number. The following statements are valid.
\begin{enumerate} [font=\normalfont, label=(\roman*)]
\item
$\displaystyle
\sum_{\substack{n \leq x\\\gcd(m,n)=1}} \frac{\mu(n)}{\varphi(n)}=O\left(e^{-c\sqrt{\log x} }\right),$
\item
$\displaystyle
\sum_{\substack{n \leq x\\\gcd(m,n)=1}} \frac{\mu(n)\log n}{\varphi(n)}=\mathfrak{G}( m)+O\left(e^{-c\sqrt{\log x} }\right),$
\end{enumerate}
where $c>0$ is a constant.
\end{lem}

\begin{lem} \label{lem4040.800P} If $x \geq 1$ is a large number, then,
\begin{equation} \label{eq4040.800P}
\sum_{\substack{n \leq x/d\\(d,n)=1}} \frac{\mu(n)\log dn}{\varphi(n)}=\mathfrak{G}( m)+O\left(e^{-b\sqrt{\log x} }\right),\nonumber
\end{equation}
where $b>0$ is a constant.
\end{lem}
\begin{proof}[\textbf{Proof}] A simple expansion of the finite sum into two finite sums and repeated applications of Lemma \ref{lem4040.850P} return

\begin{eqnarray}\label{eq4040.805P}
	\sum_{\substack{n \leq x/d\\(d,n)=1}} \frac{\mu(n)\log dn}{\varphi(n)}&=&\sum_{\substack{n \leq x/d\\ \gcd(d,n)=1}}\frac{\mu(n)\log n}{\varphi(n)}+(\log d)\sum_{\substack{n \leq x/d\\ \gcd(d,n)=1}}\frac{\mu(n)}{\varphi(n)}\\
&=&\mathfrak{G}( m)+O\left(e^{-c\sqrt{\log x} }\right)+O\left((\log d)e^{-c\sqrt{\log x} }\right) \nonumber\\
&=&\mathfrak{G}( m)+O\left(e^{-b\sqrt{\log x} }\right),\nonumber
\end{eqnarray} 
where $c>b>0$ are constants.
\end{proof}

\subsection{Sum of Twisted Log Function}
\begin{lem}\label{lem4040.600} If $x$ is a large number, then
	$$\sum_{n\leq x}\mu(n) \log n
	=O\left(xe^{-c\sqrt{\log x} }\right),$$
	where $ c>0$ is an absolute constant. 
\end{lem}
\begin{proof}[\textbf{Proof}] Recall the asymptotic formula $M(t)=\sum_{n\leq t}\mu(n)=O\left(te^{-a\sqrt{\log t} }\right)$, confer \cite[p.\ 424]{IK2004}, \cite[p.\ 385]{MV2007}. Now rewrite it as an integral and use partial summation.
	\begin{eqnarray}\label{eq4040.610}
		\sum_{n\leq x}\mu(n) \log n
		&=&\int_2^x(\log t)d M(t)\\
		&=&O\left(x(\log x)e^{-a\sqrt{\log x} }\right)-\int_2^x\frac{M(t)}{t}dt\nonumber\\
		&=&O\left(xe^{-c\sqrt{\log x} }\right),\nonumber
	\end{eqnarray}
	where $a>c>0$ are constants. 
\end{proof}
\section{\textbf{Integers in Arithmetic Progressions}}\label{A2002}
An effective asymptotic formula for the number of integers in arithmetic progressions is derived in Lemma \ref{lemA2002.400W}. The derivation is based on a version of the basic large sieve inequality stated below.
\begin{thm}\label{thmA2002.200W} Let $x$ be a large number and let $Q
	\leq x$. If $\{a_n:n\geq1\}$ is a sequence of real number, then
	\begin{equation}\label{eqA2002.100W}
		\sum_{q\leq Q}	q\sum_{1\leq a\leq q}\bigg |\sum_{\substack{n \leq x\\ n\equiv a \bmod q}}a_n-\frac{1}{q}\sum_{n \leq x}a_n\bigg|^2\leq Q\left(10Q+2\pi x \right) \sum_{n \leq x}|a_n|^2\nonumber.
	\end{equation}
\end{thm}
\begin{proof}[\textbf{Proof}] The essential technical details are covered in \cite[Chapter 23]{DH2000}. This inequality is discussed in \cite{GP1967} and the literature in the theory of the large sieve. 
\end{proof}
\begin{lem} \label{lemA2002.400W} If $x \geq 1$ is a large number and $1\leq a< q \leq x$, then
	\begin{equation}\label{eqA2002.400W}
		\sum_{q\leq x}	\max_{q}\max_{1\leq a\leq q}\bigg |\sum_{\substack{n \leq x\\ n\equiv a \bmod q}}1-\frac{1}{q}\sum_{n \leq x}1\bigg|=O\left(\frac{x}{q}e^{-c\sqrt{\log x} }\right),
	\end{equation}
	where $ c>0$ is a constant. In particular,
	\begin{equation}\label{eqA2002.405W}
		\sum_{\substack{n \leq x\\ n\equiv a \bmod q}}1=\left[\frac{x}{q}\right]+O\left(\frac{x}{q}e^{-c\sqrt{\log x} }\right).
	\end{equation}
	
\end{lem}
\begin{proof}[\textbf{Proof}] Trivially, the basic finite sum satisfies the asymptotic \begin{equation}\label{eqA2002.410W}
		\sum_{n \leq x}1=[x]= x-\{x\},
	\end{equation}
	where $[x]=x-\{x\}$ is the largest integer function, and the number of integers in any equivalent class satisfies the asymptotic formula
	\begin{equation}\label{eqA2002.415W}
		\sum_{\substack{n \leq x\\ n\equiv a \bmod q}}1	=	\frac{x}{q}+E(x).
	\end{equation}
	Let $Q=x$ and let the sequence of real numbers be $a_n=1$ for  $n\geq1$. Now suppose that the error term is of the form 
	\begin{equation}\label{eqA2002.420W}
		E(x)=E_0(x)=O\left(x^{\alpha}\right),
	\end{equation}
	where $ \alpha\in(0,1]$ is a constant. Then, the large sieve inequality, Theorem \ref{thmA2002.200W}, yields the lower bound
	\begin{eqnarray}\label{eqA2002.430W}
		\sum_{q\leq x}	q\sum_{1\leq a\leq q}\bigg |\sum_{\substack{n \leq x\\ n\equiv a \bmod q}}1-\frac{1}{q}\sum_{n \leq x}1\bigg|^2
		&=&\sum_{q\leq x}	q\sum_{1\leq a\leq q}\bigg |\frac{x}{q}+O\left(x^{\alpha}\right)-\frac{x-\{x\}}{q}\bigg|^2\nonumber\\[.2cm]
		&\gg&\sum_{q\leq x}	q\sum_{1\leq a\leq q}\bigg |x^{\alpha}+\frac{\{x\}}{q}\bigg|^2\nonumber\\[.2cm]
		&\gg&\sum_{q\leq x}	q\sum_{1\leq a\leq q}\left |x^{\alpha}\right|^2\nonumber\\[.2cm]
		&\gg&x^{2\alpha}\sum_{q\leq x}q\sum_{1\leq a\leq q}1\nonumber\\[.2cm]
		&\gg&x^{2\alpha}\sum_{q\leq x}q^2\nonumber\\[.2cm]
		&\gg&	x^{3+2\alpha} .
	\end{eqnarray}
	On the other direction, it yields the upper bound
	\begin{eqnarray}\label{eqA2002.440W}
		\sum_{q\leq x}	q\sum_{1\leq a\leq q}\bigg |\sum_{\substack{n \leq x\\ n\equiv a \bmod q}}1-\frac{1}{q}\sum_{n \leq x}1\bigg|^2
		&\leq&	 Q\left(10Q+2\pi x \right) \sum_{n \leq x}|a_n|^2\\
		&\leq&	 x\left(10x+2\pi x \right) \sum_{n \leq x}|1|^2\nonumber\\
		&\ll&	 x^3\nonumber.
	\end{eqnarray}
	Clearly, the lower bound in \eqref{eqA2002.430W} contradicts the upper bound in \eqref{eqA2002.440W}. Similarly, the other possibilities for the error term
	\begin{equation}\label{eqA2002.445W}
		E_1=O\left(\frac{x}{(\log x)^c} \right)\quad \text{ and }\quad
		E_2=O\left(xe^{-c\sqrt{\log x} }\right),
	\end{equation}
	contradict large sieve inequality. Therefore, the error term is of the form
	\begin{equation}\label{eqA2002.450W}
		E(x)=O\left(\frac{x}{q}e^{-c\sqrt{\log x} }\right)=O\left(\frac{x}{q(\log x)^c} \right)=O\left(\frac{x}{q }\right),
	\end{equation}
	where $ c>0$ is a constant.
\end{proof}

\section{\textbf{Lower Bound For The Main Term}} \label{S2099}

\begin{lem} \label{lem2099.100} If $x \geq 1$ is a large number, and $x_1=(\log x)^{c_0}\leq e^{c_1\sqrt{\log x}}$, where $c_0>0$ and $c_1=c_1(c_0)>0$ are constants, then,
	\begin{eqnarray} \label{eq2099.100}
		M(x)&=&\sum_{\substack{1\leq d_1\leq x_1\\1\leq d_2\leq x_1}}\mu(d_1)\log (d_1)\mu(d_2)\log (d_2)\sum_{\substack{n \leq x\\d_1\mid 2n+1,\;d_2\mid 2n+1}}\Lambda(n+1)\nonumber\\
		&\gg& \frac{x\log \log x}{\log \log \log x}\nonumber.
	\end{eqnarray}
\end{lem}

\begin{proof}[\textbf{Proof}] Let $x$ be a large number, let $d_1d_2\leq x_1^2=(\log x)^{2c_0}\leq e^{c_1\sqrt{\log x}}$, and let $q=[d_1,d_2]$. Applying the prime number theorem for prime in arithmetic progression, see \cite[Corollary 11.19]{MV2007}, yields	
	\begin{eqnarray} \label{eq2099.110}
		M(x)&=&\sum_{\substack{1\leq d_1\leq x_1\\1\leq d_2\leq x_1}}\mu(d_1)\log (d_1)\mu(d_2)\log (d_2)\sum_{\substack{n \leq x\\d_1\mid 2n+1,\;d_2\mid 2n+1}}\Lambda(n)\\[.2cm]
		&=&\sum_{\substack{1\leq d_1\leq x_1\\1\leq d_2\leq x_1}}\mu(d_1)\log (d_1)\mu(d_2)\log (d_2) \left( \frac{x}{\varphi([d_1,d_2])}+O\left( xe^{-c_1\sqrt{\log x}}\right) \right) \nonumber\\[.2cm]
		&=&x\sum_{\substack{1\leq d_1\leq x_1\\1\leq d_2\leq x_1}} \frac{\mu(d_1)\log (d_1)\mu(d_2)\log (d_2)}{\varphi([d_1,d_2])}\nonumber\\[.2cm]
		&&\hskip 1.5 in +	O\left( xe^{-c_1\sqrt{\log x}} \sum_{\substack{1\leq d_1\leq x_1\\1\leq d_2\leq x_1}}(\log d_1)(\log d_2)\right)\nonumber\\[.2cm]
		&=&		M_{0}(x)+M_{1}(x)\nonumber.
	\end{eqnarray}
	
	The first subsum $M_{0}(x)$ is estimated in Lemma \ref{lem2099.200} and the second subsum $M_{1}(x)$ is estimated in Lemma \ref{lem2059.300}. Summing these estimates yields
	\begin{eqnarray} \label{eq2059.130}
		M(x)&=&	M_{0}(x)+M_{1}(x)\\[.2cm]
		&\gg	& \frac{x\log \log x}{\log \log \log x}+O\left( xe^{-c_2\sqrt{\log x}}\right)\nonumber\\[.2cm]
		&\gg	&  \frac{x\log \log x}{\log \log \log x}\nonumber,
	\end{eqnarray}
	as claimed. \end{proof}

\begin{lem} \label{lem2099.200} Assume that $d_1\mid 2n+1,\;d_2\mid 2n+1$. If $x \geq 1$ is a large number, and $x_1=(\log x)^{c_0}$, where $c_0>0$, then,
	\begin{eqnarray} \label{eq2099.210}
		M_{0}(x)&=&x\sum_{\substack{1\leq d_1\leq x_1\\1\leq d_2\leq x_1}} \frac{\mu(d_1)\log (d_1)\mu(d_2)\log (d_2)}{\varphi([d_1,d_2])}\nonumber\\
		&\gg& \frac{x\log \log x}{\log \log \log x}\nonumber.
	\end{eqnarray}
\end{lem} 
\begin{proof}[\textbf{Proof}] By Lemma \ref{lem4040.250P} the double finite sum
	\begin{eqnarray} \label{eq2099.240}
		F(x)&=&\sum_{\substack{1\leq d_1\leq x_1\\1\leq d_2\leq x_1}} \frac{\mu(d_1)\log (d_1)\mu(d_2)\log (d_2)}{\varphi([d_1,d_2])}\\[.2cm]
		&\gg&\frac{\log \log x}{\log \log \log x}\nonumber.
	\end{eqnarray}
	Thus, the product $xF(x)\gg (x\log \log x)/(\log \log \log x) $ verifies the claim. 
\end{proof}

\begin{lem} \label{lem2059.300} If $x \geq 1$ is a large number, and $x_1=(\log x)^{c_0}\leq e^{c_1\sqrt{\log x}}$, where $c_0>0$ and $c_1=c_1(c_0)>0$ are constants, then,
	\begin{eqnarray} \label{eq2059.300}
		M_1(x)&=&O\left( xe^{-c_1\sqrt{\log x}} \sum_{d_1,\;d_2\leq x_1}\log (d_1)\log (d_2)\right)\nonumber\\
		&=& O\left( xe^{-c_2\sqrt{\log x}}\right)\nonumber,
	\end{eqnarray}
	where $c_1, c_2>0$ are constants.
\end{lem}
\begin{proof}[\textbf{Proof}] An estimate of the double finite sum yields
	\begin{eqnarray} \label{eq2059.310}
		M_1(x)&=&O\left( xe^{-c_1\sqrt{\log x}} \sum_{d_1,\;d_2\leq x_1}(\log d_1)(\log d_2)\right)\\[.2cm]
		&=&O\left( xe^{-c_1\sqrt{\log x}}\left((\log x_1)^2\cdot (x_1)^2 \right) \right) \nonumber\\[.2cm]
		&=&O\left( xe^{-c_1\sqrt{\log x}}(\log x)^{2c_0+1} \right) \nonumber\\[.2cm]
		&=&O\left( xe^{-c_2\sqrt{\log x}}\right)\nonumber,
	\end{eqnarray}	
	where $c_1, c_2>0$ are constants.
\end{proof}

\section{\textbf{Upper Bound For The Error Term}} \label{S2059}
The error term $E(x)$ arising in the proof of Theorem \ref{thm2020.700} consists of a sum of three finite sums
\begin{eqnarray}\label{eq2059.345}
	E(x)&=&\sum_{\substack{1\leq d_1\leq 2x+1\\ 1\leq d_2\leq 2x+1\\d_1>x_1\text{ or }d_2>x_1}}\mu(d_1)\log (d_1)\mu(d_2)\log (d_2)\sum_{\substack{n \leq x\\ d_1\mid 2n+1,\;d_2\mid 2n+1}}\Lambda(n)\nonumber\\[.2cm]
	&=&\sum_{\substack{x_1<d_1\leq 2x+1\\1\leq d_2\leq 2x+1}}+\sum_{\substack{1\leq d_1\leq 2x+1\\x_1< d_2\leq 2x+1}}+\sum_{\substack{x_1<d_1\leq 2x+1\\x_1< d_2\leq 2x+1}}\nonumber\\[.2cm]
	&=&E_1(x)+	E_2(x)+E_3(x).
\end{eqnarray}
An effective upper bound is computed in the next result.
\begin{lem} \label{lem2059.350} Assume that $d_1\mid 2n+1,\;d_2\mid 2n+1$. If $x \geq 1$ is a large number, and $x_1=(\log x)^{c_0}<xe^{c_1\sqrt{\log x}}$, with $c_0>0$, then,
	\begin{equation} \label{eq2059.350}
		E(x)= O\left( xe^{-c\sqrt{\log x}}\right)\nonumber,
	\end{equation}
	where $c_1,c>0$ are constants.	
\end{lem}
\begin{proof}[\textbf{Proof}]Except for minor changes, the analysis of the upper bounds for finite sums $E_1(x)$, $E_2(x)$ and $E_3(x)$  are similar. The first one is computed in Lemma \ref{lem2059.400} to demonstrate the method. Summing these estimates yields
	\begin{eqnarray}\label{eq2059.360}
		E(x)
		&=&E_1(x)+	E_2(x)+E_3(x)\\
		&=& O\left( xe^{-c\sqrt{\log x}}\right)\nonumber.
	\end{eqnarray}
	This completes the proof.
\end{proof}
\begin{lem} \label{lem2059.400} Assume that $d_1\mid 2n+1,\;d_2\mid 2n+1$. If $x \geq 1$ is a large number, and $x_1=(\log x)^{c_0}<xe^{c_1\sqrt{\log x}}$, with $c_0>0$, then,
	\begin{eqnarray} \label{eq2059.400}
		E_1(x)&=&\sum_{\substack{x_1<d_1\leq 2x+1\\1\leq d_2\leq 2x+1}}\mu(d_1)\log (d_1)\mu(d_2)\log (d_2)\sum_{\substack{n \leq x\\ d_1\mid 2n+1,\;d_2\mid 2n+1}}\Lambda(n)\nonumber\\
		&=& O\left( xe^{-c_3\sqrt{\log x}}\right)\nonumber,
	\end{eqnarray}
	where $c_1,c_3>0$ are constants.	
\end{lem}
\begin{proof}[\textbf{Proof}] First replace $\Lambda(n)=-\sum_{d\mid n}\mu(d)\log d$ in the inner sum.
	\begin{eqnarray} \label{eq2059.410}
		E_1(x)&=&-\sum_{\substack{x_1<d_1\leq 2x+1\\1\leq d_2\leq 2x+1}}\mu(d_1)\log (d_1)\mu(d_2)\log (d_2)\\
		&&\hskip 1in\times\sum_{\substack{n \leq x\\ d_1\mid 2n+1,\;d_2\mid 2n+1}}\sum_{d_3\mid n}\mu(d_3)\log(d_3)\nonumber\\
		&=&	-\sum_{\substack{x_1<d_1\leq 2x+1\\1\leq d_2\leq 2x+1\\1\leq d_3\leq x}}\mu(d_1)\log (d_1)\mu(d_2)\log (d_2)\mu(d_3)\log (d_3)\nonumber\\
		&&\hskip 1.97 in\times\sum_{\substack{n \leq x\\ d_1\mid 2n+1,\;d_2\mid 2n+1\\d_3\mid n}}1\nonumber.
	\end{eqnarray}	
	Next, rearrange the last finite sum in the equivalent form
	\begin{eqnarray} \label{eq2059.420}
		E_1(x)
		&=&	-\sum_{\substack{x_1<d_1\leq 2x+1\\1\leq d_2\leq 2x+1\\1\leq d_3\leq x}}\mu(d_1)\log (d_1)\mu(d_2)\log (d_2)\mu(d_3)\log (d_3) \\
		&&\hskip 1.25 in \times\left(\sum_{\substack{n \leq x\\ d_1\mid 2n+1,\;d_2\mid 2n+1\\d_3\mid n}}1-\frac{x}{[d_1,d_2,d_3]}+\frac{x}{[d_1,d_2,d_3]}\right) \nonumber\\[.4cm]
		&=&-	x\sum_{\substack{x_1<d_1\leq 2x+1\\1\leq d_2\leq 2x+1\\1\leq d_3\leq x}}\frac{\mu(d_1)\log (d_1)\mu(d_2)\log (d_2)\mu(d_3)\log (d_3)}{[d_1,d_2,d_3]}\nonumber\\[.4cm]
		&&	-\sum_{\substack{x_1<d_1\leq 2x+1\\1\leq d_2\leq 2x+1\\1\leq d_3\leq x}}\mu(d_1)\log (d_1)\mu(d_2)\log (d_2)\mu(d_3)\log (d_3)\left(\sum_{\substack{n \leq x\\ d_1\mid 2n+1,\;d_2\mid 2n+1\\d_3\mid n}}1-\frac{x}{[d_1,d_2,d_3]}\right) \nonumber\\[.4cm]
		&=& T_{0}(x)+T_{1}(x)\nonumber.
	\end{eqnarray}

	The subsum $T_{0}(x)$ is estimated in Lemma \ref{lem2059.450B} and	subsum $T_{1}(x)$ is estimated in Lemma \ref{lem2059.450}. Summing these estimates completes the proof.	 
\end{proof}

\begin{lem} \label{lem2059.450B} Assume that $d_1\mid 2n+1,\;d_2\mid 2n+1$, and $d_3\mid n$. If $x \geq 1$ is a large number, and $x_1=(\log x)^{c_0}<xe^{c_1\sqrt{\log x}}$, with $c_0>0$, then,
	\begin{eqnarray} \label{eq2059.450B}
		T_0(x)&=&-x\sum_{\substack{x_1<d_1\leq 2x+1\\1\leq d_2\leq 2x+1\\1\leq d_3\leq x}}\frac{\mu(d_1)\log (d_1)\mu(d_2)\log (d_2)\mu(d_3)\log d_3}{[d_1,d_2,d_3]}\nonumber\\
		&=& O\left( xe^{-c_3\sqrt{\log x}}\right)\nonumber,
	\end{eqnarray}
	where $c_1,c_3>0$ are constants.	
\end{lem}
\begin{proof}[\textbf{Proof}]The hypothesis $d_1\mid 2n+1,\;d_2\mid 2n+1$, and $d_3\mid n$ implies that $[d_1,d_2,d_3]=[d_1,d_2]d_3$ since $\gcd(d_1d_2,d_3)=1$. Thus, the finite sum can be factored as
	\begin{eqnarray} \label{eq2059.460B}
		T_{0}(x)
		&=&-	x\sum_{\substack{x_1<d_1\leq 2x+1\\1\leq d_2\leq 2x+1\\1\leq d_3\leq x}}\frac{\mu(d_1)\log (d_1)\mu(d_2)\log (d_2)\mu(d_3)\log (d_3)}{[d_1,d_2,d_3]}\\[.2cm]
		&=&x\sum_{\substack{x_1<d_1\leq 2x+1\\1\leq d_2\leq 2x+1}}\frac{\mu(d_1)\log (d_1)\mu(d_2)\log (d_2)}{[d_1,d_2]}\sum_{1\leq d_3\leq x}\frac{\mu(d_3)\log (d_3)}{d_3} \nonumber.
	\end{eqnarray}	
	Applying Lemma \ref{lem4040.600} to the inner sum in  \eqref{eq2059.460B} and Lemma \ref{lem4040.800L} to the middle sum, yield
	\begin{eqnarray}\label{eq2059.470B}
		T_{0}(x)
		&=& O\bigg( x\sum_{\substack{x_1<d_1\leq 2x+1\\1\leq d_2\leq 2x+1}}\frac{\log (d_1)\log (d_2)}{[d_1,d_2]}\left( e^{-c_2\sqrt{\log x}}\right) \bigg)\\
		&=& O\left( x(\log x)^5\left( e^{-c_2\sqrt{\log x}}\right) \right)\nonumber\\	&=& O\left( xe^{-c_3\sqrt{\log x}}\right)\nonumber,
	\end{eqnarray}	
	where $c_1,c_3>0$ are constants.
\end{proof}

\begin{lem} \label{lem2059.450} Assume that $d_1\mid 2n+1,\;d_2\mid 2n+1$, and $d_3\mid n$. If $x \geq 1$ is a large number, and $x_1=(\log x)^{c_0}<xe^{c_1\sqrt{\log x}}$, with $c_0>0$, then,
	\begin{eqnarray} \label{eq2059.450}
		T_1(x)&=&\sum_{\substack{x_1<d_1\leq 2x+1\\1\leq d_2\leq 2x+1\\1\leq d_3\leq x}}\mu(d_1)\log (d_1)\mu(d_2)\log (d_2)\mu(d_3)\log (d_3)\left(\sum_{\substack{n \leq x\\ d_1\mid 2n+1,\;d_2\mid 2n+1\\d_3\mid n}}1-\frac{x}{[d_1,d_2,d_3]}\right)\nonumber\\
		&=&O\left( xe^{-c_4\sqrt{\log x}}\right)\nonumber,
	\end{eqnarray}
	where $c,c_1,c_4>0$ are constants.
\end{lem}
\begin{proof}[\textbf{Proof}] Let $q=[d_1,d_2,d_3]$. Taking absolute value and invoking Lemma \ref{lemA2002.400W} yield
	\begin{eqnarray} \label{eq2059.455}
|T_1(x)|&\leq&\sum_{\substack{x_1<d_1\leq 2x+1\\1\leq d_2\leq 2x+1\\1\leq d_3\leq x}}\log (d_1)\log (d_2)\log (d_3)\left|\sum_{\substack{n \leq x\\ d_1\mid 2n+1,\;d_2\mid 2n+1\\d_3\mid n}}1-\frac{x}{[d_1,d_2,d_3]}\right|\nonumber\\
&\ll&\sum_{\substack{x_1<d_1\leq 2x+1\\1\leq d_2\leq 2x+1\\1\leq d_3\leq x}}\log (d_1)\log (d_2)\log (d_3)\left(\frac{x}{[d_1,d_2,d_3]}e^{-c\sqrt{\log x} } \right) \nonumber\\
&\ll&xe^{-c\sqrt{\log x} }\sum_{\substack{x_1<d_1\leq 2x+1\\1\leq d_2\leq 2x+1\\1\leq d_3\leq x}}\frac{\log (d_1)\log (d_2)\log (d_3)}{[d_1,d_2,d_3]}  \nonumber.
\end{eqnarray}	
	
The hypothesis $d_1\mid 2n+1,\;d_2\mid 2n+1$, and $d_3\mid n$ implies that $[d_1,d_2,d_3]=[d_1,d_2]d_3$ since $\gcd(d_1d_2,d_3)=1$. Thus, the finite sum can be factored as
	\begin{equation} \label{eq2059.460}
		T_1(x) =O\bigg( x e^{-c\sqrt{\log x}} \sum_{\substack{x_1<d_1\leq 2x+1\\1\leq d_2\leq 2x+1}}\frac{\log (d_1)\log (d_2)}{[d_1,d_2]}\sum_{1\leq d_3\leq x}\frac{\log (d_3)}{d_3}\bigg) \nonumber.
	\end{equation}	
	Estimating the inner sum, and applying Lemma \ref{lem4040.800L} to the middle sum return
	\begin{eqnarray} \label{eq2059.465}
		T_1(x) 
		&=&O\left( x e^{-c\sqrt{\log x}} \cdot(\log x)^5\cdot(\log x)^2\right)  \\
		&=& O\left(x(\log x)^7) e^{-c\sqrt{\log x}}\right) \nonumber\\
		&=& O\left( xe^{-c_4\sqrt{\log x}}\right)\nonumber,
	\end{eqnarray}	
	where $c,c_1,c_4>0$ are constants.
\end{proof}

\section{\textbf{Fundamental Results}} \label{S2020}The classical weighted Germain primes counting function has the form
\begin{equation}\label{eq2020.900}
	\sum_{1\leq n \leq x} \Lambda(n)\Lambda(2n+1).
\end{equation}
The derivation of a lower bound for the number of Germain primes up to a large number $x\geq1$ is based on a new weighted Germain primes counting function
\begin{equation}\label{eq2020.905}
	\sum_{ n \leq x} w(n)\Lambda(n)\Lambda(2n+1).
\end{equation}
The extra weight factor $w(n)=\Lambda(2n+1)$ provides effective control over the error term at the cost of a smaller main term, by a factor of approximately $\log x$.

\begin{thm}\label{thm2020.700}
	If $x\geq 1$ is a large real number, then
	\begin{equation}
		\sum_{1\leq n\leq x}w(n)\Lambda(n)\Lambda(2n+1) \gg \frac{x\log \log x}{\log \log \log x}\nonumber.
	\end{equation}
\end{thm}
\begin{proof}[\textbf{Proof}] Substitute the identity $\Lambda(n)=-\sum_{d\mid n}\mu(d)\log d$, see \cite[Theorem 2.11]{AT1976}, then reverse the order of summations.
	\begin{eqnarray} \label{eq2020.910}
		\psi_{0}(x)&=&\sum_{n \leq x} \Lambda(n)\Lambda^2(2n+1)\\[.2cm]
		&=&\sum_{n \leq x} \Lambda(n)\sum_{d_1\mid 2n+1}\mu(d_1)\log (d_1)\sum_{d_2\mid 2n+1}\mu(d_1)\log (d_2)\nonumber\\[.2cm]
		&=&\sum_{\substack{1\leq d_1\leq 2x+1\\ 1\leq d_2\leq 2x+1}}\mu(d_1) \log (d_1)\mu(d_2)\log (d_2)\sum_{\substack{n \leq x\\ d_1\mid 2n+1,\;d_2\mid 2n+1}}\Lambda(n)\nonumber.
	\end{eqnarray} 
	Let $x_1=(\log x)^{c_0}$, with $c_0>0$ constant, and partition the triple finite sum. 
	\begin{eqnarray} \label{eq2020.920}
		\psi_{0}(x)
		&=&	\sum_{\substack{1\leq d_1\leq x_1\\ 1\leq d_2\leq x_1}}\mu(d_1) \log (d_1)\mu(d_2)\log (d_2)\sum_{\substack{n \leq x\\ d_1\mid 2n+1,\;d_2\mid 2n+1}}\Lambda(n)\\[.2cm]
		&&+\sum_{\substack{1\leq d_1\leq 2x+1\\ 1\leq d_2\leq 2x+1\\d_1>x_1\text{ or }d_2>x_1}}\mu(d_1) \log (d_1)\mu(d_2)\log (d_2)\sum_{\substack{n \leq x\\ d_1\mid 2n+1,\;d_2\mid 2n+1}}\Lambda(n)\nonumber\\[.2cm]
		&=&M(x)+E(x)\nonumber.
	\end{eqnarray} 
	Summing the main term computed in Lemma \ref{lem2099.100} and the error term computed in Lemma \ref{lem2059.350}, yields
	\begin{eqnarray} \label{eq2020.930}
		\psi_0(x)
		&=&M(x)+E(x)\\[.2cm]
		&\gg&\frac{x\log \log x}{\log \log \log x}+O\left( xe^{-c_1\sqrt{\log x}}\right)\nonumber	\\[.2cm]
		&\gg&\frac{x\log \log x}{\log \log \log x},\nonumber
	\end{eqnarray}
	where $c_1>0$ is a constant. 
\end{proof}

\begin{thm}\label{thm2020.950}
	If $x\geq 1$ is a large real number, then
	\begin{equation}
		\sum_{ n\leq x}\Lambda(n)\Lambda(4n+1) \gg \frac{x\log \log x}{(\log x)(\log \log \log x)}\nonumber.
	\end{equation}
\end{thm}
\begin{proof}[\textbf{Proof}] The proof of this result is the same as the previous result in Theorem \ref{thm2020.700} mutuatis mutandis.
\end{proof}

\section{\textbf{The Main Results}} \label{S3535}

\begin{proof}[\textbf{Proof}] (Theorem \ref{thm3030.900}) Partial summation, and an application of Theorem \ref{thm2020.700} yield
	\begin{eqnarray}\label{eq3535.900}
		\psi_{\mathcal{G}}(x)&=&	\sum_{ n \leq x} \Lambda(n)\Lambda(2n+1)\\[.2cm]
		&\gg&	\sum_{ n \leq x} \frac{w(n)\Lambda(n)\Lambda(2n+1)}{\log n}\nonumber\\[.2cm]
		&\gg&\int_2^x\frac{1}{\log z}d\psi_0(z)\nonumber\\[.2cm]
		&\gg&\frac{x\log \log x}{(\log x)(\log \log \log x)}.\nonumber
	\end{eqnarray}	
	Quod erat inveniendum.
\end{proof}

\section{\textbf{Sums over the Germain Primes}}  \label{S3288}
A few of the basic standard sums over the subset of Germain primes 
\begin{equation}\label{eq3288.150}
\mathcal{G}=\{p:2p+1 \text{  is prime }\}=\{2,3,5,11,23,\cdots \}
\end{equation} 
are computed in this section.
\begin{lem} \label{lem3208.150}If $x\geq1$ is a large number, then
	\begin{equation}\label{eq3288.150b}
		\sum_{\substack{p\leq x\\p\in \mathcal{G}}}	\frac{1}{p}	=\frac{1}{2}+\frac{1}{3}+\frac{1}{5}+\frac{1}{11}+\frac{1}{23}+\cdots\geq1.167720685111989459
	\end{equation}
converges to a constant.
\end{lem}
An expanded and more precise calculation is compiled in \cite[Table 1]{WS2021}.

\begin{lem} \label{lem3288.160}If $\mathcal{G}=\{2,3,5,11,\cdots \}$ is the subset of Germain primes, and $x\geq1$ is a large number, then
	\begin{equation}\label{eq3288.160}
		\sum_{\substack{p\leq x\\p\in \mathcal{G}}}	\frac{\log p}{p}	=	a_0\log\log x+\frac{a_0}{\log x}+O\left(\frac{1}{(\log x)^2} \right),
	\end{equation}
where $a_0=2C_2>0$.	
\end{lem}
\begin{proof}[\textbf{Proof}] Set $\pi_{\mathcal{G}}(x)=a_0x/(\log x)^2(1+o(1))$ and evaluate the integral representation:
	\begin{eqnarray}\label{eq3288.160b}
	\sum_{\substack{p\leq x\\p\in \mathcal{G}}}	\frac{\log p}{p}&=&	\int_1^{x}\frac{\log t}{t}d\pi_{\mathcal{G}}(t) \\[.2cm]
	&=&	a_0\log\log x+\frac{a_0}{\log x}+O\left(\frac{1}{(\log x)^2} \right) \nonumber.
\end{eqnarray}	
The constant $a_0=2C_2>0$ is the same as \eqref{eq3030.900}.
\end{proof}

\section{\textbf{Application}} \label{S3545}

\begin{thm}\label{thm3545.950} The integer $2$ is a primitive root modulo a Germain prime $4p+1$ infinitely often. Moreover, if $x\geq 1$ is a large real number, the weighted counting function of such primes has the lower bound 
	\begin{equation}
		\sum_{ n\leq x}\Lambda(n)\Lambda(4n+1) \gg \frac{x\log \log x}{(\log x)(\log \log \log x)}\nonumber.
	\end{equation}
\end{thm}
\begin{proof}[\textbf{Proof}] The existence of infinitely many Germain prime pairs $p$ and $4p+1$ follow from Theorem \ref{thm2020.950}. To show that 2 is a primitive root infinite often, fix a prime $q=4p+1>16$. By Lemma \ref{lem3252.550}, the integer $2$ is a primitive root modulo $q$ if and only if
	
\begin{equation}\label{eq3545.950b}
2^{\frac{q-1}{2}}	\equiv \left ( \frac{2}{q}\right )\equiv(-1)^{\frac{q^2-1}{8}}\equiv(-1)^{\frac{16p^2+8p}{8}}\equiv-1\not \equiv 1 \mod  q,
\end{equation}	
see Lemma \ref{lem3242.850}, and
\smallskip\noindent
\begin{equation}\label{eq3545.950c}
	2^{\frac{q-1}{p}}	\equiv  2^4\not \equiv 1 \mod  q.
\end{equation}		
This completes the verification.
\end{proof}

A small table of prime pairs $p$ and $q=4p+1$ is recorded below. The integer $2$ is automatically a primitive root modulo $q$.
\begin{table}[ht]
	\setlength{\tabcolsep}{0.5cm}
	\renewcommand{\arraystretch}{1.250092}
	\setlength{\arrayrulewidth}{.6pt}
	\centering
\begin{tabular}{c|c|c|c|c|c|c|c}
$p$	&  $4p+1$&  $p$& $4p+1$&$p$& $4p+1$&$p$& $4p+1$\\
	\hline
3	&  13& 79 & 317&277&1109&577& 2309\\
	\hline
7	&  29& 97 &  389&307& 1229&619&2477\\
	\hline
13	& 53 & 127& 509&373& 1493&673&2697\\
	\hline
37	& 149 &  139& 557&409&1637 &709&2837\\
	\hline
43	&173  & 163 &  653&433&1733&739&2959\\
	\hline
67	& 269 & 193 &773 &487&1949 &853&3413\\
	\hline
73	& 293 & 199 & 797&499&1997&883&3533 \\

\end{tabular}
	\vskip .1 in	
\caption{A list of small prime pairs $p$ and $4p+1$ }
\label{table2}
\end{table}
\section{\textbf{Appendix}} \label{A3242}

\subsection{Primitive Roots Tests}  \label{S3242.100}
For any prime $p \geq 3$, the multiplicative group $G$ of the prime finite fields $\mathbb{F}_p$ is a cyclic group of cardinality $p-1=\#G$. Similar result is true for any finite extension $\mathbb{F}_q$ of $\mathbb{F}_p$, where $q=p^k$ is a prime power. \\

\begin{dfn} \label{dfn3242.100} {\normalfont The \textit{multiplicative order} of an element $u\ne0$ in the cyclic group $\mathbb{F}_p^\times$ is defined by $\ord_p(u)=\min \{k \in \mathbb{N}: u^k \equiv 1 \bmod p \}$. An element is a \textit{primitive root} if and only if $\ord_p(u)=p-1$. 
	}
\end{dfn}

The Euler totient function counts the number of relatively prime integers \(\varphi (n)=\#\{ k:\gcd (k,n)=1 \}\). This counting function is compactly expressed by
the analytic formula \(\varphi (n)=n\prod_{p \mid n}(1-1/p),n\in \mathbb{N} .\)\\

\begin{lem} {\normalfont (Fermat-Euler)} \label{lem3242.200}If \(a\in \mathbb{Z}\) is an integer such that \(\gcd (a,n)=1,\) then \(a^{\varphi (n)}\equiv
	1 \bmod n\).
\end{lem}

\begin{lem} \label{lem3242.500}  {\normalfont (Primitive root test in $\mathbb{F}_p$)} An integer $u \in \Z$ is a primitive root modulo an integer $n \in \N$ if and only if 
	\begin{equation}\label{eq3242.500}
		u^{\varphi(n)/p} -1\not \equiv 0 \mod  n
	\end{equation}
	for all prime divisors $p \mid \varphi(n)$.
\end{lem}
The primitive root test is a special case of the Lucas primality test, introduced in \cite[p.\ 302]{ LE1878}. A more recent version appears in \cite[Theorem 4.1.1]{CP2005}, and similar sources. \\

\begin{lem} \label{lem3242.850} The integer $2$ is a quadratic residue, {\normalfont (}quadratic nonresidue{\normalfont )} of the primes of the form $p=8k\pm1$, {\normalfont (}$p=8k\pm3$ respectively{\normalfont )}. Equivalently,
	\begin{equation}\label{eq3242.850}
		\left ( \frac{2}{p}\right )=(-1)^{\frac{p^2-1}{8}}.
	\end{equation}
\end{lem}
\begin{proof}A detailed proof of the quadratic reciprocity laws appears in \cite[Theorem 1.5.]{RH1994}, and similar references.
\end{proof}
\begin{lem} \label{eq3242.860}  {\normalfont (Quadratic reciprocity law)} If $p$ and $q$ are odd primes, then
	\begin{equation}\label{eq3242.870}
		\left ( \frac{p}{q}\right )\left ( \frac{q}{p}\right )=(-1)^{\frac{p-1}{2}\frac{q-1}{2}}.
	\end{equation}
\end{lem}
\begin{proof}A detailed proof of the quadratic reciprocity laws appears in \cite[Theorem 2.1.]{RH1994}, and similar references.
\end{proof}

\subsection{Very Short Primitive Roots Tests}  \label{S3252.200}
The set of Fermat primes 
\begin{equation}\label{lem3252.150}
	\mathcal{F}=\left \{F_n=2^{2^n}+1:n\geq0\right \}= \left \{3,5,17,257,65537,\ldots\right \}
\end{equation}
has the simplest primitive root test: A quadratic nonresidue $q\geq3$ is a primitive root mod $F_n$. This follows from Lemma \ref{lem3242.500}. The next set of primes with a short primitive root test seems to be the set of generalized Germain primes.
\begin{dfn} \label{dfn3252.500} {\normalfont Let $s\geq 1$ be a parameter. The set of generalized Germain primes is defined by 
		\begin{equation} \label{eq3252.340}
			\mathcal{G}_s=\{p=2^s\cdot r+1: p \text{ and } r \text{ are primes}\}=\{3, 5, 7,13,17,23,29,37,43,41,73,\ldots\}\nonumber.
		\end{equation}
	}
\end{dfn}

\begin{lem} \label{lem3252.550} An integer $q\ne\pm1, n^2$ is a primitive root modulo a Germain prime $p=2^s\cdot r+1$ if and only if 
	\begin{multicols}{2}
		\begin{enumerate}[font=\normalfont, label=(\roman*)]
			\item $\displaystyle  q^{2^{s-1}r}\not \equiv 1 \mod  p$,      
			\item $\displaystyle q^{2^{s}}\not \equiv 1 \mod  p.$
		\end{enumerate}
	\end{multicols}
\end{lem}

\begin{proof} Let $p=2^sr+1$ be a Germain prime, where $r\geq 2$ is prime, and $s\geq1$ is an integer. Since the totient $p-1=2^sr$ has two prime divisors, an integer $q\ne\pm1, n^2$ is a primitive root modulo $p$ if and only if 
	\begin{enumerate}[font=\normalfont, label=(\roman*)]
		\item $\displaystyle  q^{(p-1)/2}=q^{2^{s-1}r}\not \equiv 1 \mod  p$,
		\item $\displaystyle q^{(p-1)/r} =q^{2^{s}}\not \equiv 1 \mod  p,$
	\end{enumerate}
	see Lemma \ref{lem3242.500}.
\end{proof}

As the Germain primes, the set of primes of the form
\begin{equation}
	\mathcal{A}=\{p=k\cdot 2^n+1:n\geq0\},    
\end{equation}
with $k\geq3$ a fixed prime, have a very short primitive root test, similar to the algorithm in Lemma \ref{lem3252.550}. Some of these primes are factors of Fermat numbers. There are many interesting problems associated with these primes, a large literature, and numerical data, see \cite{KW1983}, et cetera.


\end{document}